\documentclass[a4paper] {article}
[12pt]

\makeatletter
\renewcommand*\l@section{\@dottedtocline{1}{1.5em}{2.3em}}
\makeatother

\usepackage{amsfonts}
\usepackage{amssymb}
\usepackage[T1]{fontenc}

\usepackage{tikz}
\usetikzlibrary{calc}

\usepackage{CJK}
\usepackage{amsmath}
 
\usepackage{amsfonts}
\usepackage{amssymb}
\usepackage{amsthm}
\usepackage{amssymb}
\usepackage{enumerate}
\usepackage[calc]{picture}
\usepackage[all,cmtip]{xy}

\usepackage[mathscr]{eucal}
\usepackage{eqlist}

\usepackage{color}
\usepackage{abstract} 
\usepackage[T1]{fontenc}
 
\usepackage[top=3cm, bottom=3cm, left=3.8cm, right= 3.8cm]{geometry}

\theoremstyle{plain}
\newtheorem{theorem}{Theorem}
\newtheorem{proposition}[theorem]{Proposition}
\newtheorem{lemma}[theorem]{Lemma}

\newtheorem{example}[theorem]{Example}
\newtheorem{corollary}[theorem]{Corollary}

\theoremstyle{definition}

\usepackage{etoolbox}
\newtheoremstyle{myrem}
 {3pt}
 {3pt}
 {\normalsize}
 { }
 {\itshape}
 {:}
 { }
 {}

 \theoremstyle{myrem}
 \newtheorem{remark}{Remark}
 \appto\remark{\leftskip\parindent}
 \appto\remark{\rightskip\parindent}

\numberwithin{equation}{section}
\numberwithin{theorem}{section}

\begin{document}

\begin{center}
{\Large {\textbf {Weighted Simplicial Complexes and Weighted Analytic Torsions 
}}}
 \vspace{0.58cm}\\

Shiquan Ren*,     Chengyuan Wu*

\footnotetext[1] { *first authors.  }

\bigskip

\begin{quote}
\begin{abstract}
A weighted simplicial complex  is a simplicial complex with values (called weights) on the vertices.  In this paper,  we consider weighted simplicial complexes with $\mathbb{R}^2$-valued weights.  We study the weighted homology and the  weighted analytic torsion for such weighted simplicial complexes. 
\end{abstract}
\end{quote}

\end{center}

\section{Introduction}

 During the 20-th century, 
the theory of R-torsion and  analytic torsion of Riemannian manifolds was developed by K. Reidemeister \cite{rei}, John Milnor \cite{milnor},  D. B. Ray and I. Singer \cite{ray}, Werner M\"{u}ller \cite{muller},  J. Cheeger \cite{ch},   etc.   
Recently, Alexander Grigor'yan,  Yong  Lin  and  Shing-Tung Yau \cite{lin1}   applied the R-torsion and the analytic torsion theory to digraphs and path complexes (cf. \cite{lin2,lin3,lin4,lin5,lin6}).  A discrete version of the R-torsion and the analytic torsion theory on digraphs has been given in \cite{lin1}.

Simplicial complexes are useful  combinatorial models in  algebraic topology.  In recent decades, (abstract) simplicial complexes have been found to have various applications in data sciences.  Let $V$ be a finite set with a total order $\prec$.   An {\it (abstract)  simplicial complex}  $\mathcal{K}$ on $V$ is a subset of the power set $2^V$ such that for any $\sigma\in\mathcal{K}$  and any non-empty subset $\tau\subseteq\sigma$,   we have that  $\tau\in\mathcal{K}$.  An element $\sigma\in\mathcal{K}$  is called a {\it simplex},  and an element $v\in V$  is called a {\it vertex}.  For any simplex $\sigma\in\mathcal{K}$,  we can always write $\sigma$  as $\{v_0,v_1,\ldots, v_n\}$  for some non-negative integer $n$,  where $v_0,v_1,\ldots, v_n\in V$  and $v_0\prec v_1\prec\ldots  \prec v_n$.  We write $n=\dim\sigma$ and call it the {\it dimension}  of $\sigma$.  Letting $\sigma$ run over all the simplices in $\mathcal{K}$, the maximum of $\dim \sigma$ is denoted as $\dim\mathcal{K}$ and is called the dimension of $\mathcal{K}$.

 A weighted simplicial complex  is a simplicial complex with some functions that  assign values to the vertices (or to  the simplices).  In 1990,  Robert. J. MacG.  Dawson \cite{1990}  studied the (weighted) homology for  weighted simplicial complexes.  In recent years, the weighted homology of  weighted simplicial complexes  have been further explored  in \cite{chengyuan1,chengyuan11}.  Some other topological features for weighted simplicial complexes, such as the fundamental groups,  the discrete Morse functions,  the  Hodge-Laplace operators, etc. have been studied in \cite{chengyuan2,chengyuan3, chengyuan5}.  In 2020,  Zhenyu Meng, D Vijay Anand, Yunpeng Lu, Jie Wu  and Kelin Xia \cite{xia1} have found amazing applications of   the weighted homology  for weighted simplicial complexes  in boilogical data analysis, with significant effects.

\smallskip

In this paper,  we consider the  weighted simplicial complexes $(\mathcal{K},f,g)$ where both $f$ and $g$ are real functions on the set $V$ of the vertices.  We use $f$ to twist the boundary operators and use $g$ to give a symmetric and semi-positive definite quadratic form on the chain complexes.  We study the weighted homology  equipped with the symmetric and semi-positive definite quadratic forms.  Then we study the weighted analytic torsions  for weighted simplicial complexes.

For any real number $r$,  we let $\epsilon(r)=1$ if $r\neq 0$ and let $\epsilon(r)=0$ if $r=0$.  Then for any function $f$ on $V$,  $\epsilon(f)$ is a function on $V$  given by $\epsilon(f)(v)=\epsilon(f(v))$ for any $v\in V$.   For any real-valued function $f$ on $V$,  let $\mathcal{K}^\times _f$  be the largest sub-simplicial complex of $\mathcal{K}$  such that $f$ is non-vanishing on all of its vertices.  We will prove the following theorem.  

\begin{theorem}[Main Result]
\label{th-main}
Let $\mathcal{K}$  be a simplicial complex with the set $V$  of its vertices.  Let $(f,g)$ be a $\mathbb{R}^2$-valued function on $V$. Let $T(\mathcal{K},f,g)$ be the $(f,g)$-weighted analytic torsion of $\mathcal{K}$.  Then 
\begin{enumerate}[(i).]
\item
for any non-vanishing  real function $h$ on $V$,  we have 
\begin{eqnarray*}
T(\mathcal{K},\epsilon(g) fh,\epsilon(f) gh)=T(\mathcal{K}, \epsilon(g)f,\epsilon(f)g);
\end{eqnarray*} 
\item
 for any non-zero real constant $c$ we have
\begin{eqnarray*}
T(\mathcal{K}, \epsilon(g)f,c\epsilon(f)g)=|c|^{s(\mathcal{K},\epsilon(f)\epsilon(g))}T(\mathcal{K},\epsilon(g)f,\epsilon(f)g)
\end{eqnarray*}
  and 
\begin{eqnarray*}
T(\mathcal{K},c\epsilon(g)f,\epsilon(f)g)=|c|^{-s(\mathcal{K},\epsilon(f)\epsilon(g))} T(\mathcal{K},\epsilon(g)f,\epsilon(f)g). 
\end{eqnarray*}
Here  
\begin{eqnarray*}
s(\mathcal{K},\epsilon(f)\epsilon(g))=\sum_{n\geq 0}(-1)^n \dim \partial_n C_n(\mathcal{K}^\times_{\epsilon(f)\epsilon(g)};\mathbb{R}).
\end{eqnarray*} 
\end{enumerate}
\end{theorem}

As a by-product,  we prove in Theorem~\ref{pr-3.19.g}  that for any  weighted simplicial complex $(\mathcal{K},f,g)$ and  any $n\geq 0$,  there is a symmetric and semi-positive definite quadratic form $\langle~,~\rangle_g$ on the $n$-th weighted homology $H_n(\mathcal{K},f,g;\mathbb{R})$. In addition,  if $g$ is non-vanishing, then  the quadratic form $\langle~,~\rangle$  is an inner product.  We also prove in Theorem~\ref{th-isom}  and Corollary~\ref{co-isom} that for any non-vanishing real function $h$  on $V$,  there is a linear  isometry from $H_n(\mathcal{K},fh,gh;\mathbb{R})$  to $H_n(\mathcal{K},f,g;\mathbb{R})$.  Thus  the $(f,g)$-weighted homology $H_n(\mathcal{K},f,g;\mathbb{R})$ only depends on the ratio function  $f/g$ for any non-vanishing functions $f$ and $g$.   In particular,  if $f=g$, then there is a  linear isometry from the weighted homology $H_n(\mathcal{K},f,f;\mathbb{R})$ to the usual (unweighted) homology $H_n(\mathcal{K};\mathbb{R})$  (cf.  \cite[Section~5]{chengyuan11}).

\smallskip

 The remaining part of this paper is organized as follows.  In Section~\ref{s2},  we give some prelinimaries in linear algebra. In Section~\ref{s3},  we study the weighted homology of weighted simplicial complexes.  We prove Theorem~\ref{th-isom} and Theorem~\ref{pr-3.19.g}.  In Section~\ref{s4},  we study the weighted analytic torsions for  weighted simplicial complexes and  prove Theorem~\ref{th-main}.   

\smallskip

\section{Preliminaries}\label{s2}

Let $W$  be a (finite dimensional) real vector space.  Let $\langle~,~\rangle$  be a symmetric semi-positive definite quadratic form on $W$.  Consider the sub-space 
\begin{eqnarray*}
N=\{v\in W\mid  \langle v,v\rangle=0\}=\{ v\in W\mid \langle v,u\rangle=0{\rm ~for~any~} u\in W\}.
\end{eqnarray*}
Note that the second equality is obtained by using the Cauchy-Schwarz inequality.  We call $N$ the {\bf null sub-space} for  the quadratic form $\langle~,~\rangle$. The next lemma proves that the quotient space of $W$  by $N$  has an inherited inner product.

\begin{lemma}\label{le-1.1}
The quotient space $W/N$ has an inner product $\langle~,~\rangle$ which is inherited from the quadratic form $\langle~,~\rangle$ on $W$. 
\end{lemma}

  \begin{proof}
For any $v_1+ N, v_2+ N\in  W/N$,  we define 
\begin{eqnarray}\label{eq-1.2}
\langle v_1+ N, v_2+ N\rangle =\langle v_1,v_2\rangle. 
\end{eqnarray}
To prove (\ref{eq-1.2})  is well-defined,  we let $v'_1=v_1+n_1$ and $v'_2=v_2+ n_2$ where $n_1,n_2\in N$.  By the Cauchy-Schwarz inequality, 
\begin{eqnarray}\label{eq-1.3}
|\langle v_1,n_2\rangle|^2\leq  \langle v_1,v_1\rangle \langle n_2,n_2\rangle. 
\end{eqnarray}
Since $n_2\in N$, $\langle n_2,n_2\rangle=0$.  Hence (\ref{eq-1.3}) implies
$\langle v_1,n_2\rangle=0$. 
Similarly,  we have
$\langle n_1,v_2\rangle=0$ and 
$\langle n_1,n_2\rangle=0$.  Therefore,   we have (\ref{eq-1.2}),  and the inherited quadratic form $\langle~,~\rangle$  is well-defined on $W/N$. 

To prove $\langle~,~\rangle$  is an inner product on $W/N$,  we need to verify that it is strictly positive-definite. Suppose 
$\langle v+N,  v+N\rangle=0$.  Then $\langle v,v\rangle=0$,  which implies $v\in N$.  Therefore, $\langle~,~\rangle$  is strictly positive-definite  on $W/N$,  thus it is an inner product. 
\end{proof}

As a genralization of Lemma~\ref{le-1.1},  the next lemma shows that any quotient space of $W$ inherits a quadratic form. 

\begin{lemma}\label{le-1.q}
For any sub-space $U$ of $W$,  the quotient space $W/U$  inherits a symmetric semi-positive definite quadratic form $\langle~,~\rangle$ from $W$.  Moreover, $\langle~,~\rangle$  is an inner product on $W/U$ if and only if $N\subseteq U$. 
\end{lemma}

\begin{proof}
Let $U$ be a sub-space of $W$.  Restricted to $U$, the quadratic form $\langle~,~\rangle$ on $W$ gives a quadratic form $\langle~,~\rangle$ on $U$,  which is still symmetric and semi-positive definite. By Lemma~\ref{le-1.1},  $U/(U\cap N)$  inherits an inner product $\langle~,~\rangle$.  Note that as vector spaces,  
\begin{eqnarray}\label{eq-1.11}
U\cong U/(U\cap N) \oplus (U\cap N). 
\end{eqnarray}
 We let $\iota$  be the isomorphism from the right-hand side of (\ref{eq-1.11}) to the left-hand side of (\ref{eq-1.11})  and  let $U_1$ be the image $\iota(U/(U\cap N))$. Then  
\begin{eqnarray}\label{eq-1.12}
U=U_1 \oplus (U\cap N). 
\end{eqnarray}
We point out that the isomorphism $\iota$ may not be unique hence the sub-space $U_1$ may not be unique as well.  Suppose both $\iota$ and $U_1$ are fixed.  For any $u\in U$,  with respect to (\ref{eq-1.12}) we can write $u=u_1+ n$ uniquely where $u_1\in U_1$ and $n\in U\cap N$.  Then for any $u,u'\in U$ we have 
\begin{eqnarray}\label{eq-1.13}
\langle u,u'\rangle=\langle u_1,u'_1\rangle. 
\end{eqnarray}
The left-hand side of (\ref{eq-1.13}) is a semi-positive definite quadratic form and the right-hand-side of (\ref{eq-1.13}) is an inner product.  By a similar argument of the Gram-Schmidt process,  we can extend $U_1$ to be a sub-space $W_1$ of $W$ such that  
\begin{eqnarray}\label{eq-1.16}
W= W_1\oplus N 
\end{eqnarray}
as vector spaces where $W_1$ inherits an inner product $\langle~,~\rangle$.  We consider the orthogonal complement $\perp_{W_1}U_1$ of $U_1$ in $W_1$.  Then 
\begin{eqnarray}\label{eq-1.17}
W_1=(\perp_{W_1}U_1) \oplus U_1 
\end{eqnarray}
as Euclidean spaces. It follows from (\ref{eq-1.16}) and (\ref{eq-1.17}) that for any $v\in W$,  we  can write 
\begin{eqnarray}\label{eq-1.18}
v=v_0+ v_2 +n
\end{eqnarray}
where $v_0\in U_1$,  $v_2\in \perp_{W_1}U_1$ and $n\in N$.  Moreover,  once $W_1$ is fixed, the expression (\ref{eq-1.18}) is unique.  
With the help of (\ref{eq-1.18}) we define the quadratic form on $W/U$  by setting 
\begin{eqnarray}\label{eq-1.19}
\langle v+ U,v'+U\rangle =\langle v_2,v'_2\rangle  
\end{eqnarray}
for any $v,v'\in W$.  Note that the right-hand side of (\ref{eq-1.19}) does not depend on the representatives $v$  of $v+U$ and $v'$ of $v'+U$.   Hence (\ref{eq-1.19}) gives a well-defined quadratic form on $W/U$, which is obviously symmetric and semi-positive definite.

We point out that $W_1$,  the extension of $U_1$,     may not be unique.  Nevertheless,  for any two such extensions  $W_1$ and   $\bar{W}_1$ of $U_1$,  if we write the corresponding decompositions in (\ref{eq-1.18})  as $v=v_0+ v_2 +n$  and $v= \bar{v}_0 + \bar{v}_2 + \bar{n}$ respectively,  then for any $v\in W$,  $v_0=\bar{v}_0$; and for any $v,v'\in V$,  
\begin{eqnarray}\label{eq-1.20}
\langle v_2,v'_2\rangle =\langle \bar{v}_2,\bar{v}'_2\rangle. 
\end{eqnarray}
It follows from (\ref{eq-1.20}) that the inherited quadratic form on $W/U$ given  by (\ref{eq-1.19})  does not depend on the extension $W_1$.   We obtain the first assertion.

By  the above argument,  it follows from  (\ref{eq-1.12})  and (\ref{eq-1.16})    that 
\begin{eqnarray*}
W/U\cong \big(\perp_{W_1}U_1\big)\oplus \big(N/(U\cap N)\big). 
\end{eqnarray*} 
 We  see that $\langle~,~\rangle$  is an inner product on $W/U$  if and only if it is strictly positive definite on $W/U$, which happens  if and only if 
\begin{eqnarray}\label{eq-1.588}
W/U\cong \perp_{W_1}U_1.
\end{eqnarray}
Therefore,   (\ref{eq-1.588})  holds if and only if 
\begin{eqnarray*}
N/(U\cap N)=0,
\end{eqnarray*}
which happens if and only if $N\subseteq U$. We obtain the second assertion.  
\end{proof}

With the help  of  Lemma~\ref{le-1.q},  we have the next lemma for linear maps between vector spaces with quadratic forms.   

\begin{lemma}\label{le-1.qq}
Let $W$  and $W'$  be vector spaces with (symmetric and semi-positive definite)  quadratic forms $\langle~,~\rangle$ and $\langle~,~\rangle'$  respectively.  Let $\varphi: W\longrightarrow W'$ be a linear isomorphism  such that for any $a,b\in W$,  $\langle \varphi(a),\varphi(b)\rangle'=\langle a,b\rangle$.  Let $U\subset W$ and $U'\subset W'$  be subspaces   such that $\varphi(U)=U'$.  Then $\varphi$  induces  a linear isomorphism 
\begin{eqnarray*}
\varphi_*: W/U\longrightarrow W'/U'
\end{eqnarray*}
such that  the induced  symmetric and semi-positive defnite quadratic forms on $W/U $ and $W'/U'$  are preserved.  
\end{lemma}

\begin{proof}
It  is clear that $\varphi$  is  a linear isomorphism.  By Lemma~\ref{le-1.q},  there are  induced  symmetric and semi-positive defnite quadratic forms $\langle~,~\rangle$  on $W/U$  and $\langle~,~\rangle'$  on $W'/U'$.  Moreover,  $\varphi$  preserves the decompositions (\ref{eq-1.16})  and (\ref{eq-1.17}).  We verify that $\varphi$ preserves the quadratic forms  by
\begin{eqnarray*} 
\langle v+U, w+U\rangle&=&\langle v_2,w_2\rangle\\
&=&\langle \varphi(v_2),\varphi(w_2)\rangle'\\
&=&\langle \varphi(v+U), \varphi(w+U)\rangle'. 
\end{eqnarray*}
Here $ v+U$  and  $w+U$  are any two elements in $W/U$,  and $v_2$  and $w_2$ are in $\perp _{W_1}U_1$ given by (\ref{eq-1.18}). We obtain the lemma. 
\end{proof}

\smallskip

\section{Weighted Simplicial Complexes and Weighted Homology}
\label{s3}

Let $\mathcal{K}$  be a simplicial complex.  
Let $V$ be the set of the vertices of $\mathcal{K}$.  
Let $f$  and $g$ be two real-valued functions on  $V$.  
Then 
\begin{eqnarray*}
(f,g):  V\longrightarrow \mathbb{R}^2
\end{eqnarray*}  
is  a  vector-valued  function 
assigning a point in the plane 
to each vertex in $V$.   
 We call both $f$ and $g$ a {\bf  vertex-weight}  
on $\mathcal{K}$ and call the ordered triple  $(\mathcal{K},f,g)$ 
 a {\bf  vertex-weighted simplicial complex}.

Let $n\geq 0$.  
For  any $n$-simplex $\sigma=\{v_0,v_1,\ldots,v_n\}$ in $\mathcal{K}$,  
we define  the induced weights of $f$ and $g$  on $\sigma$  respectively  as 
\begin{eqnarray*}
f(\sigma)=\prod_{i=0}^n f(v_i), ~~~ g(\sigma)=\prod_{i=0}^n g(v_i). 
\end{eqnarray*} 
Let $C_n(\mathcal{K};\mathbb{R})$  be the vector space 
 consisting of all the formal linear combinations 
 of the $n$-simplices in $\mathcal{K}$.   
The $n$-th  {\bf $f$-weighted boundary  operator} is  a linear map 
\begin{eqnarray*}
\partial_n^f: C_n(\mathcal{K};\mathbb{R})
\longrightarrow C_{n-1}(\mathcal{K};\mathbb{R})
\end{eqnarray*}
 defined by 
\begin{eqnarray}\label{eq-wbo}
\partial_n^f\{v_0,v_1,\ldots, v_n\}=
\sum_{i=0}^n(-1)^i f(v_i) \{v_0,\ldots, \widehat{v_i}, \ldots, v_n\}
\end{eqnarray}
for any $n$-simplex $\{v_0,v_1,\ldots, v_n\}$  
of $\mathcal{K}$.   Particularly,  for any $n$-simplex $\sigma$  in $\mathcal{K}$  with  $f(\sigma)\neq 0$,  (\ref{eq-wbo})  can be alternatively expressed as 
\begin{eqnarray}\label{eq-wbo1}
\partial_n^f \sigma=\sum_{i=0}^n (-1)^i \frac{f(\sigma)}{f(d_i\sigma)} d_i\sigma
\end{eqnarray}
where for each $0\leq i\leq n$,  $d_i\sigma$  is the $i$-th  $(n-1)$-face of $\sigma$  obtained   by removing the $i$-th vertex of $\sigma$.  Note that $f(\sigma)\neq 0$  implies $f(d_i\sigma)\neq 0$  for each $0\leq i\leq n$.  Hence (\ref{eq-wbo1}) makes sense.  
It can be verified that 
\begin{eqnarray*}
\partial^f_{n}\partial^f_{n+1}=0
\end{eqnarray*}
 for any $n\geq 0$.  Thus 
\begin{eqnarray}\label{eq-wbo2}
\{C_n(\mathcal{K};\mathbb{R}),  \partial_n^f \}_{n\geq 0}
\end{eqnarray}
  is a  chain complex.

 The  {\bf $g$-weighted quadratic form} on $C_n(\mathcal{K};\mathbb{R})$  is  a symmetric  and  semi-positive definite quadratic form 
\begin{eqnarray}\label{eq-wip}
\langle~,~\rangle_g:  C_n(\mathcal{K};\mathbb{R})\times C_n(\mathcal{K}; \mathbb{R})\longrightarrow\mathbb{R}
\end{eqnarray}
given by 
\begin{eqnarray*}
\langle \sigma,\tau\rangle_g &=& g(\sigma)g(\tau) \delta(\sigma,\tau)\\
&=&\prod_{i=0}^n  g(v_i)g(u_i)\delta(v_i,u_i)
\end{eqnarray*}
for any  $n$-simplices  $\sigma=\{v_0,v_1,\ldots, v_n\}$   and  $\tau=\{u_0,u_1,\ldots, u_n\}$  in $\mathcal{K}$.   Here for any vertices $v$  and $u$,  we use the notation $\delta(v,u)=0$  if $v\neq u$  and $\delta(v,u)=1$  if $v=u$;  and  for any simplices $\sigma$ and $\tau$,  we use the notation  $\delta(\sigma,\tau)=0$  if $\sigma\neq \tau$  and $\delta(\sigma,\tau)=1$  if $\sigma=\tau$.

We notice that  for each $n\geq 0$,  (\ref{eq-wip}) gives a  $g$-weighted  quadratic form on the  vector space in the chain complex (\ref{eq-wbo2}).  We  use the following notation 
\begin{eqnarray}\label{eq-wcc}
C_*(\mathcal{K},f,g;\mathbb{R})= \{C_n(\mathcal{K};\mathbb{R}),\partial_n^f,\langle~,~\rangle_g\}_{n\geq 0}
\end{eqnarray}
to denote the chain complex  (\ref{eq-wbo2})  with the $f$-weighted boundary operators $\partial_n^f$ and the $g$-weighted quadratic forms $\langle~,~\rangle_g$.    
 The  $n$-th {\bf $(f,g)$-weighted homology} of $\mathcal{K}$ (with coefficients in real numbers) is defined as  the $n$-th homology group of the chain complex   $C_*(\mathcal{K},f,g;\mathbb{R})$,  i.e.  the quotient space    
\begin{eqnarray*}
H_n(\mathcal{K},f,g;\mathbb{R})={\rm Ker}(\partial_n^f)/{\rm Im}(\partial_{n+1}^f). 
\end{eqnarray*}
The next theorem says that for each $n\geq 0$,     the quadratic form  $\langle~,~\rangle_g$  on  $C_n(\mathcal{K}; \mathbb{R})$  induces a  quadratic form,  which is still denoted as  $\langle~,~\rangle_g$,  on the $n$-th $(f,g)$-weighted homology group $H_n(\mathcal{K},f,g;\mathbb{R})$. 

\begin{theorem}\label{pr-3.19.g}
For any vertex-weighted simplicial complex $(\mathcal{K},f,g)$ and any $n\geq 0$, 
$H_n(\mathcal{K},f,g;\mathbb{R})$  is a vector space with  a symmetric and semi-positive definite quadratic form $\langle~,~\rangle_g$.  
\end{theorem}

\begin{proof}
In Lemma~\ref{le-1.q},  we let the space $W$  be  ${\rm Ker}\partial_n^f$ and let the sub-space $U$  of $W$  be ${\rm Im}\partial_{n+1}^f$.  As sub-spaces of $C_n(\mathcal{K};\mathbb{R})$,  both $W$  and $U$ have an inherited  symmetric and semi-positive definite quadratic forms  $\langle~,~\rangle_g$.   Thus by Lemma~\ref{le-1.q},  the quotient space $W/U$, which by definition is the $(f,g)$-weighted homology $H_n(\mathcal{K},f,g;\mathbb{R})$,  has an induced symmetric and semi-positive definite quadratic form $\langle~,~\rangle_g$. 
\end{proof}

For each $n\geq 0$,   the null sub-space of $C_n(\mathcal{K};\mathbb{R})$ with respect to $\langle~,~\rangle_g$  is given by 
\begin{eqnarray*}
N_n(\mathcal{K},g;\mathbb{R})&=&\{  b\in C_n(\mathcal{K};\mathbb{R})\mid \langle b,b\rangle_g=0\} \\
&=&\{  b\in C_n(\mathcal{K};\mathbb{R})\mid \langle b,a\rangle_g=0{\rm ~for~any~} a\in C_n(\mathcal{K};\mathbb{R})\}. 
\end{eqnarray*}
We denote  $N_*(\mathcal{K},g;\mathbb{R})$  for the graded vector space $\{N_n(\mathcal{K},g;\mathbb{R})\mid n=0,1,2,\ldots\}$.   By applying the  second assertion in  Lemma~\ref{le-1.q}  to  the proof of Theorem~\ref{pr-3.19.g},  the next corollary follows   
 immediately.  

\begin{corollary}
For any vertex-weighted simplicial complex $(\mathcal{K},f,g)$ and any $n\geq 0$, 
 $\langle~,~\rangle_g$  is an inner product on $H_n(\mathcal{K},f,g;\mathbb{R})$ if and only if   $N_n(\mathcal{K},g;\mathbb{R})$  is a sub-space of $\partial^f_{n+1} C_n(\mathcal{K};\mathbb{R})$.  \qed
\end{corollary}

For each vertex $v\in V$,  we  let ${\rm star}_{\mathcal{K}}(v)$  be the open star of $v$ in $\mathcal{K}$ given by the set of the simplices containing $v$ as a vertex: 
\begin{eqnarray*}
{\rm star}_\mathcal{K}(v)=\{\sigma\in \mathcal{K}\mid v\in\sigma\}. 
\end{eqnarray*} 
We have the next lemma.

\begin{lemma}\label{le-2.3}
Let $(\mathcal{K},f,g)$  be a  vertex-weighted simplicial complex. Then for any simplex  $\sigma\in \mathcal{K}$,  
\begin{eqnarray}\label{eq-2.8}
g(\sigma)=0 \Longleftrightarrow \sigma\in \bigcup_{v\in V,\atop g(v)=0 } {\rm~star}_\mathcal{K}(v).
\end{eqnarray}
Thus  for each $n\geq 0$,  $N_n(\mathcal{K},g;\mathbb{R})$  has  a  basis consisting of all the $n$-simplices in the union of open stars 
\begin{eqnarray}\label{eq-2.11}
\bigcup_{v\in V,\atop g(v)=0 } {\rm~star}_\mathcal{K}(v).
\end{eqnarray} 
\end{lemma}

\begin{proof}
Let $\sigma\in\mathcal{K}$.  
We write $\sigma$ in the form $\{v_0,v_1,\ldots,v_n\}$ for some $n\geq 0$.  Since $f$ is a vertex-weight,  it follows from the definition of $f(\sigma)$   that $f(\sigma)=0$  if and only if $f(v_i)=0$  for some $0\leq i\leq n$.  This happens if and only if $\sigma$ is in the open star of $v_i$  for some $v_i\in V$ with  $f(v_i)=0$.  Thus we obtain (\ref{eq-2.8}).  Consequently, by the definition  of $N_n(\mathcal{K},g;\mathbb{R})$,  it  follows from (\ref{eq-2.8})  that for each $n\geq 0$,  $N_n(\mathcal{K},g;\mathbb{R})$ has a basis  consisting    of  all  the $n$-simplices in (\ref{eq-2.11}).  We obtain  the lemma. 
\end{proof}

The following lemma shows that if we take the quotient space of $N_n(\mathcal{K},g;\mathbb{R})$  in $C_n (\mathcal{K};\mathbb{R})$,  then 
we will obtain the chain complex of the maximal  sub-simplicial complex $\mathcal{K}^\times_g$ of $\mathcal{K}$  where $g$  is non-vanishing.

\begin{lemma}\label{le-2.3y}
Let $(\mathcal{K},f,g)$  be a  vertex-weighted simplicial complex.  Then the sub-simplicial complex 
\begin{eqnarray*} 
\mathcal{K}_g^\times=\{ \sigma\in\mathcal{K}\mid g(\sigma)\neq 0\}
\end{eqnarray*}
is the maximal sub-simplicial complex of $\mathcal{K}$ such that $g$  is non-vanishing on all of its vertices.  Moreover, for each $n\geq 0$,   with respect to the inner products  $\langle~,~\rangle_g$,     there  a canonical  linear isometry  
\begin{eqnarray}\label{eq-3.19-1}
\iota_n: C_n(\mathcal{K}^\times_g;\mathbb{R})\longrightarrow  C_n(\mathcal{K};\mathbb{R})/N_n(\mathcal{K},g;\mathbb{R}). 
\end{eqnarray}
\end{lemma}

\begin{proof}
Let $(\mathcal{K},f,g)$  be a  vertex-weighted simplicial complex.   It follows from Lemma~\ref{le-2.3}  that the maximal sub-simplicial complex $\mathcal{K}_g^\times$  such that $g$  is non-vanishing on all of its vertices is the complement of the union  (\ref{eq-2.11})  of open stars 
\begin{eqnarray*}
\mathcal{K}_g^\times =\mathcal{K}\setminus \Big(\bigcup_{v\in V, \atop g(v)=0} {\rm~star}_\mathcal{K}(v)\Big).  
\end{eqnarray*}
 Consequently,  for each $n\geq 0$,  we have a linear isomorphism $\iota_n$ in (\ref{eq-3.19-1}) sending an $n$-simplex $\sigma\in\mathcal{K}^\times_g$  to the co-set $\sigma+N_n(\mathcal{K},g;\mathbb{R})$.  By Lemma~\ref{le-1.1},    for each $n\geq 0$, the quotient space $C_n(\mathcal{K};\mathbb{R})/N_n(\mathcal{K},g;\mathbb{R})$ has an induced inner product $\langle~,~\rangle_g$.  Since $\mathcal{K}_g^\times $  is a sub-simplicial complex of $\mathcal{K}$,  $C_n(\mathcal{K}^\times_g;\mathbb{R})$  is a sub-space of $C_n(\mathcal{K};\mathbb{R})$ thus it inherits a quadratic form $\langle~,~\rangle_g$.   Note that $\langle~,~\rangle_g$ is strictly positive-definite  on $C_n(\mathcal{K}^\times_g;\mathbb{R})$  hence   it  is an inner product on $C_n(\mathcal{K}^\times_g;\mathbb{R})$.   For any simplices $\sigma$ and $\tau$ in $\mathcal{K}^\times_g$,  it can be verified that $\iota_n$ preserves the inner products:
\begin{eqnarray*}
\langle \iota_n(\sigma),\iota_n(\tau)\rangle_g&=& \langle \sigma+N_n(\mathcal{K},g;\mathbb{R}), \tau+N_n(\mathcal{K},g;\mathbb{R})\rangle_g\\
&=& \langle \sigma+\mu, \tau+\nu\rangle_g\\
&=&\langle\sigma,\tau\rangle_g+ \langle \sigma,\nu\rangle_g+\langle \mu,\tau\rangle_g +\langle \mu,\nu\rangle_g\\
&=&\langle \sigma,\tau\rangle_g. 
\end{eqnarray*}
Here  $\mu$ and $\nu$ are arbitrary elements  in $N_n(\mathcal{K},g;\mathbb{R})$.  Thus $\iota_n$ is a linear  isometry.  
\end{proof}

By Lemma~\ref{le-2.3y},  for each $n\geq 0$ we can define the $f$-weighted {\bf reduced boundary operator}   
\begin{eqnarray}\label{eq-3.22-1}
\tilde {\partial}^f_n: C_n(\mathcal{K};\mathbb{R})/N_n(\mathcal{K},g;\mathbb{R})\longrightarrow C_n(\mathcal{K};\mathbb{R})/N_n(\mathcal{K},g;\mathbb{R})
\end{eqnarray}
as 
\begin{eqnarray}\label{eq-3.22-2}
\tilde{\partial}^f_n= \iota_n\circ(\partial_n^f\mid_{\mathcal{K}_g^\times})\circ \iota_n^{-1}. 
\end{eqnarray}
Here 
\begin{eqnarray}\label{eq-3.22-3}
\partial_n^f\mid_{\mathcal{K}_g^\times}: C_n(\mathcal{K}^\times_g;\mathbb{R})\longrightarrow C_{n-1}(\mathcal{K}^\times_g;\mathbb{R})
\end{eqnarray}
  is the $n$-th $f$-weighted boundary operator of $\mathcal{K}^\times _g$, which is obtained by restricting the $f$-weighted boundary operator $\partial_n^f$   of $C_n(\mathcal{K};\mathbb{R})$  to  the  sub-chain complex  $C_n(\mathcal{K}^\times_g;\mathbb{R})$.    
Similar with (\ref{eq-wcc}),  we use the following notations   
\begin{eqnarray*}
C_*(\mathcal{K}^\times_g,f,g;\mathbb{R})=\{C_n(\mathcal{K}^\times_g;\mathbb{R}), \partial^f_n, \langle~,~\rangle_g\}_{n\geq 0}
\end{eqnarray*}
and
\begin{eqnarray*}
\tilde C_*(\mathcal{K},f,g;\mathbb{R})=\{C_n(\mathcal{K};\mathbb{R})/N_n(\mathcal{K},g;\mathbb{R}), \tilde \partial^f_n, \langle~,~\rangle_g\}_{n\geq 0} 
\end{eqnarray*}
to denote the chain complexes with the $f$-weighted boundary operators $\partial_n^f$  (or $\tilde\partial^f_n$)  and  the $g$-weighted inner products $\langle~,~\rangle_g$.  We  have the next proposition.

\begin{proposition}\label{pr-zzz}
Let $(\mathcal{K},f,g)$  be a vertex-weighted simplicial complex.  Then 
we  have an isomorphism of chain complexes 
\begin{eqnarray*}
\iota_*: \tilde C_*(\mathcal{K},f,g;\mathbb{R})\longrightarrow C_*(\mathcal{K}^\times_g,f,g;\mathbb{R}) 
\end{eqnarray*}
where  for each $n\geq 0$, $\iota_n$  is a linear isometry with respect to the corresponding inner products $\langle~,~\rangle_g$.   
\end{proposition}
\begin{proof}
The proof follows by the definition of $\iota_*$  in Lemma~\ref{le-2.3y}. 
\end{proof}

By taking the homology of the chain complexes,  the isomorphism $\iota_*$  induces an isomorphism of homology groups. The next corollary follows from Proposition~\ref{pr-zzz} immediately. 

\begin{corollary}\label{co-3.19.x}
Let $(\mathcal{K},f,g)$  be a vertex-weighted simplicial complex.  Then for each $n\geq 0$,  the weighted homology $H_n(\mathcal{K}_g^\times, f,g;\mathbb{R})$ is isometrically isomorphic to the  homology $H_n(\tilde C_*(\mathcal{K},f,g;\mathbb{R}))$ of the chain complex   $\tilde C_*(\mathcal{K},f,g;\mathbb{R})$   with respect to the  inner products  $\langle~,~\rangle_g$.  \qed
\end{corollary}

\smallskip

In the remaining part of this section,  we study the effects of $f$  and $g$ on the weighted homology $H_n(\mathcal{K}, f,g;\mathbb{R})$.  We prove the next lemma by using an argument similar with \cite[Lemma~3.1]{dig1}.

\begin{lemma}\label{le-5.a}
Let $\mathcal{K}$  be a simplicial complex with the set $V$ of its vertices.  Let  $f$ and  $g$   be any two real-valued functions on $V$.  Let $h$ a non-vanishing real-valued function on $V$.   Then  there is a canonical chain isomorphism 
\begin{eqnarray}\label{eq-mod}
\varphi: C_*(\mathcal{K},f,g;\mathbb{R})\longrightarrow  C_*(\mathcal{K},fh,gh;\mathbb{R}) 
\end{eqnarray}
given by 
\begin{eqnarray}\label{eq-rep}
\varphi\Big(  \sum  x_{\{v_0,v_1,\ldots, v_n\}} \{v_0,v_1,\ldots, v_n\} \Big)=\sum  \frac{x_{\{v_0,v_1,\ldots, v_n\}}}{h(v_0)h(v_1)\cdots h(v_n)} \{v_0,v_1,\ldots, v_n\} 
\end{eqnarray}
such that
\begin{eqnarray*}
\langle\varphi(a),\varphi(b)\rangle_{gh}=\langle a,b\rangle_g
\end{eqnarray*}
for any $n$-chains $a$  and  $b$ in $C_*(\mathcal{K},f,g;\mathbb{R})$. 
\end{lemma}

\begin{proof}
 Let $n\geq 0$.   We denote an $n$-simplex of $\mathcal{K}$  as $\{v_0,v_1,\ldots,v_n\}$ where the vertices $v_i$ are distinct for $i=0,1,\ldots,n$  and $v_i\prec v_j$  for $0\leq i<j\leq n$ with respect to the total order $\prec$ on $V$.

Firstly,  we  take  
\begin{eqnarray*}
\sum  x_{\{v_0,v_1,\ldots, v_n\}} \{v_0,v_1,\ldots, v_n\}
\end{eqnarray*}
 to be any $n$-chain in $C_n(\mathcal{K};\mathbb{R})$.  Then  we have 
\begin{eqnarray*}
&&\partial_n^{fh}\circ \varphi \Big(\sum  x_{\{v_0,v_1,\ldots, v_n\}} \{v_0,v_1,\ldots, v_n\}\Big)\\
&=&\partial_n^{fh} \Big(\sum  \frac{x_{\{v_0,v_1,\ldots, v_n\}}}{h(v_0)h(v_1)\cdots  h(v_n)} \{v_0,v_1,\ldots, v_n\}\Big)\\
&=&\sum  \frac{x_{\{v_0,v_1,\ldots, v_n\}}}{h(v_0)h(v_1)\cdots h(v_n)} \sum_{i=1}^n (-1)^i  f(v_i)h(v_i) \{v_0,\ldots, \widehat{v_i},\ldots, v_n\}\\
&=&\sum x_{\{v_0v_1\ldots v_n\}}\sum _{i=0}^n (-1)^i f(v_i)\frac{\{v_0,\ldots,  \widehat{v_i}, \ldots, v_n\}}{h(v_0)\cdots \widehat{h(v_i)}\cdots h(v_n)}\\
&=&\varphi\circ \partial_n^f\Big(\sum  x_{v_0,v_1,\ldots, v_n} \{v_0,v_1,\ldots, v_n\}\Big). 
\end{eqnarray*}
Thus $\varphi$ is a chain map.

Secondly,  since $h$  is non-vanishing on $V$,  it follows from (\ref{eq-rep})  directly that $\varphi$ is a linear isomorphism from $C_n(\mathcal{K};\mathbb{R})$  to itself.

 Thirdly,  let 
\begin{eqnarray*}
\sum  x_{\{v_0,v_1,\ldots, v_n\}} \{v_0,v_1,\ldots, v_n\} {\rm ~ and ~} \sum  y_{\{u_0,u_1,\ldots, u_n\}} \{u_0,u_1,\ldots, u_n\}
\end{eqnarray*}
   be any two of the $n$-chains in $C_n(\mathcal{K};\mathbb{R})$.   Then 
\begin{eqnarray*}
&&\Big\langle\varphi\Big(\sum  x_{\{v_0,v_1,\ldots, v_n\}} \{v_0,v_1, \ldots, v_n\}\Big),\varphi\Big(\sum  y_{\{u_0,u_1,\ldots, u_n\}} \{u_0,u_1,\ldots, u_n\}\Big)\Big\rangle_{gh}\\
&=&\sum\sum  \frac{x_{\{v_0,v_1,\ldots, v_n\}}}{h(v_0)h(v_1)\cdots h(v_n)}\cdot  \frac{y_{\{u_0,u_1,\ldots, u_n\}}}{h(u_0) h(u_1)\cdots h(u_n)}\langle\{v_0,v_1,\ldots, v_n\},\{u_0,u_1,\ldots, u_n\}\rangle_{gh}\\
&=&\sum\sum  \frac{x_{\{v_0,v_1,\ldots, v_n\}}}{h(v_0)h(v_1)\cdots h(v_n)}\cdot  \frac{y_{\{u_0,u_1,\ldots, u_n\}}}{h(u_0) h(u_1)\cdots h(u_n)}\prod_{i=0}^n g(v_i)h(v_i) g(u_i)h(u_i)\delta(v_i,u_i)\\
&=&\sum\sum  x_{\{v_0,v_1,\ldots, v_n\}}y_{\{u_0,u_1,\ldots, u_n\}} \prod_{i=0}^n g(v_i)g(u_i) \delta(v_i,u_i)\\
&=&\Big\langle\sum  x_{\{v_0,v_1,\ldots, v_n\}} \{v_0,v_1,\ldots, v_n\},  \sum  y_{\{u_0,u_1,\ldots, u_n\}} \{u_0,u_1,\ldots, u_n\}\Big\rangle_g. 
\end{eqnarray*}
Thus $\varphi$  preserves the quadratic forms with respect to $\langle~,~\rangle_{gh}$  and $\langle~,~\rangle_g$.

Summarizing all the above three steps,  we obtain the lemma. 
\end{proof}

The next corollary follows directly  by applying Lemma~\ref{le-5.a} to  the sub-simplicial  complex  $\mathcal{K}^\times_g$.  

\begin{corollary}\label{co-3.25.1}
Let $\mathcal{K}$  be a simplicial complex with the set $V$ of its vertices.  Let  $f$ and  $g$   be any two real-valued functions on $V$.  Let $h$ a non-vanishing real-valued function on $V$.   Then  there is a canonical chain isomorphism 
\begin{eqnarray}\label{eq-modxxx}
\varphi: C_*(\mathcal{K}^\times_g,f,g;\mathbb{R})\longrightarrow  C_*(\mathcal{K}^\times_g,fh,gh;\mathbb{R}) 
\end{eqnarray}
given by (\ref{eq-rep})
which is an isometry with respect to the inner products $\langle~,~\rangle_g$ and $\langle~,~\rangle_{gh}$ respectively. 
\end{corollary}
\begin{proof}
Note that $g$  is non-vanishing on the set of the vertices of $\mathcal{K}^\times _g$.  Thus $\langle~,~\rangle_g$  is an inner product on $C_n(\mathcal{K}^\times_g;\mathbb{R})$ for each $n\geq 0$.  Moreover,  since $h$  is non-vanishing,   $\langle~,~\rangle_{gh}$  is an inner product on $C_n(\mathcal{K}^\times_g;\mathbb{R})$  as well for each $n\geq 0$.  
\end{proof}

With the help of Lemma~\ref{le-5.a},   we  have the next theorem.

\begin{theorem}\label{th-isom}
Let $(\mathcal{K},f,g)$  be a vertex-weighted simplicial complex where $f$ and $g$ are any real-valued functions on $V$.   Let $h$  be a non-vanishing real function on $V$.   Then for each $n\geq 0$,  we have a linear isomorphism
\begin{eqnarray}\label{eq-3.19-8}
\varphi_*: H_n(\mathcal{K},f,g;\mathbb{R})\longrightarrow H_n(\mathcal{K},fh,gh;\mathbb{R})
\end{eqnarray}
which preserves the quadratic forms with respect to  $\langle~,~\rangle_g$ and $\langle~,~\rangle_{gh}$.  Here $\varphi_*$  is induced by the chain isomorphism $\varphi$  in Lemma~\ref{le-5.a}.   
\end{theorem}

\begin{proof}
Since $\varphi$ in Lemma~\ref{le-5.a} is a chain isomorphism that preserves the quadratic forms,    it follows with the help of  Lemma~\ref{le-1.qq} that  $\varphi_*$ induces a linear isomorphism $\varphi_*$  from the  $n$-th  $(f,g)$-weighted homology group $ H_n(\mathcal{K},f,g;\mathbb{R})$  to  the $n$-th $(fh,gh)$-weighted homology group $H_n(\mathcal{K},fh,gh;\mathbb{R})$ that preserves the quadratic forms on the quotient spaces.  The theorem follows.    
\end{proof}

\begin{remark}
Particularly,  if both $f$  and $g$ are non-vanishing on $V$,  then Theorem~\ref{th-main} implies that the $(f,g)$-weighted homology $H_*(\mathcal{K},f,g;\mathbb{R})$  (as graded vector space with quadratic forms $\langle~,~\rangle_g$) only depends on the ratio function $f/g$  (or $g/f$ equivalently).  
\end{remark}

Similar with Theorem~\ref{th-isom},  the next corollary follows from Corollary~\ref{co-3.25.1}.  

\begin{corollary}\label{co-isom}
Let $(\mathcal{K},f,g)$  be a vertex-weighted simplicial complex.  Then for each $n\geq 0$ and each non-vanishing function $h$ on the set of the vertices,  we  have  a linear isometry
\begin{eqnarray}\label{eq-3.19-8}
\varphi_*: H_n(\mathcal{K}^\times_g,f,g;\mathbb{R})\longrightarrow H_n(\mathcal{K}^\times_g,fh,gh;\mathbb{R}). 
\end{eqnarray}
In particular,  taking $f=g$ and $h=1/f$,   we have for each $n\geq 0$, 
\begin{eqnarray*}
H_n(\mathcal{K}^\times_f,f,f;\mathbb{R})\cong H_n(\mathcal{K}^\times_{\epsilon(f)};\mathbb{R})
\end{eqnarray*}
 as Euclidean spaces. 
\qed 
\end{corollary}

\smallskip

\section{Weighted Analytic Torsions for Weighted Simplicial Complexes}
\label{s4}

Throughout this section,  we
let $(\mathcal{K},f,g)$  be a vertex-weighted simplicial complex where  $\mathcal{K}$  is  a simplicial complex  and  $f$  and $g$ are arbitrary real-valued functions on the set $V$ of the vertices of $\mathcal{K}$.

For each $n\geq 0$,  we let $\langle~,~\rangle_g$  be the  inner product on $C_n(\mathcal{K};\mathbb{R})/N_n(\mathcal{K},g;\mathbb{R})$ induced by $g$.  
Let 
\begin{eqnarray*}
(\tilde \partial^{f}_n)^*_g: C_{n-1}(\mathcal{K};\mathbb{R})/N_n(\mathcal{K},g;\mathbb{R})\longrightarrow C_n(\mathcal{K};\mathbb{R})/N_n(\mathcal{K},g;\mathbb{R})
\end{eqnarray*}
  be the adjoint linear map of $\tilde \partial_n^f$ (which is given in (\ref{eq-3.22-1}) and (\ref{eq-3.22-2})) with respect to the inner product $\langle~,~\rangle_g$ such that 
\begin{eqnarray*}
\langle\tilde \partial^f_n (a),b \rangle_g= \langle a,(\tilde \partial^{f}_n)^*_g (b)\rangle_g
\end{eqnarray*}
for any $n$-chains $a$ and $b$ in $\tilde C_*(\mathcal{K},f,g;\mathbb{R})$.  The $n$-th $(f,g)$-weighted (reduced) Hodge-Laplace operator (on the quotient chain complex) is the linear map
\begin{eqnarray*}
\tilde \Delta_n^{f,g}:  \tilde C_n(\mathcal{K};\mathbb{R})\longrightarrow \tilde C_n(\mathcal{K};\mathbb{R})
\end{eqnarray*}
given by 
\begin{eqnarray*}
\tilde\Delta_n^{f,g}(c)=(\tilde \partial^{f}_n)^*_g\tilde \partial^{f}_n  (c)+  (\tilde \partial^{f}_{n+1})^*_g\tilde \partial^{f}_{n+1}(c)  
\end{eqnarray*}
for any $n$-chains $c$ in $\tilde C_*(\mathcal{K},f,g;\mathbb{R})$.

On the other hand, for each $n\geq 0$,  we also use $\langle~,~\rangle_g$  to denote the inner product on $C_n(\mathcal{K}^\times_g;\mathbb{R})$  with respect to $g$.   Let 
\begin{eqnarray*}
(\partial^{f}_n\mid_{\mathcal{K}^\times _g})^*_g:C_{n-1}(\mathcal{K}^\times_g;\mathbb{R})\longrightarrow C_n(\mathcal{K}^\times_g;\mathbb{R})
\end{eqnarray*}
  be the adjoint linear map of $ \partial_n^f\mid_{\mathcal{K}^\times _g}$ (which is given in (\ref{eq-3.22-3})) with respect to the inner products $\langle~,~\rangle_g$   such that 
\begin{eqnarray*}
\langle \partial^f_n (a),b \rangle_g= \langle a,(\partial^{f}_n\mid_{\mathcal{K}^\times_g})^*_g (b)\rangle_g
\end{eqnarray*}
for any $n$-chains $a$ and $b$ in $C_*(\mathcal{K}^\times_g,f,g;\mathbb{R})$.  The $n$-th $(f,g)$-weighted Hodge-Laplace operator of $\mathcal{K}^\times_g$ is the linear map
\begin{eqnarray*}
\Delta_n^{f,g}:  C_n(\mathcal{K}^\times_g;\mathbb{R})\longrightarrow  C_n(\mathcal{K}^\times_g;\mathbb{R})
\end{eqnarray*}
given by 
\begin{eqnarray*}
\Delta_n^{f,g}(c)=( \partial^{f}_n\mid_{\mathcal{K}^\times _g})^*_g \partial^{f}_n  (c)+  (  \partial^{f}_{n+1}\mid_{\mathcal{K}^\times _g})^*_g \partial^{f}_{n+1}(c)  
\end{eqnarray*}
for any $n$-chains $c$ in $C_*(\mathcal{K}^\times_g,f,g;\mathbb{R})$.

The reduced Hodge-Laplace operator $\Delta_n^{f,g}$ and the Hodge Laplace  operator $\Delta_n^{f,g}$  can be identified via the isomorphism $\iota_n$.  In fact,  it follows from Proposition~\ref{pr-zzz} and \cite[Lemma~2.2]{dig1}  that  for each $n\geq 0$,  
\begin{eqnarray*}
\Delta_n^{f,g}=\iota_n\circ \tilde\Delta_n^{f,g}\circ  (\iota_n)^{-1}. 
\end{eqnarray*}
Since $\iota_n$  is an isometry,  it follows that counted with multiplicities,  the multi-set of the eigenvalues of $\Delta_n^{f,g}$ and the multi-set of the eigenvalues of $\tilde\Delta_n^{f,g}$ are equal.  For simplicity,  we do not distinguish $\tilde\Delta_n^{f,g}$  and $\Delta_n^{f,g}$.  We define the $n$-th {\bf $(f,g)$-weighted Hodge-Laplace operator}   of the vertex-weighted simplicial complex $(\mathcal{K},f,g)$  as 
$\Delta_n^{f,g}$.

We observe that  $\Delta_n^{f,g}$  is symmetric and semi-positive definite.  Hence its matrix representatives are diagonalizable and its eigenvalues are non-negative.  We use $\lambda$ to denote any eigenvalue of $\Delta_n^{f,g}$.  The zeta function $\zeta_n(s)$  is defined by
\begin{eqnarray*}
\zeta_n(s)=\sum_{\lambda>0} \frac{1}{\lambda^s}. 
\end{eqnarray*}
The $(f,g)$-weighted analytic torsion of  $(\mathcal{K},f,g)$  is given by 
\begin{eqnarray}\label{eq-t-def}
\log T(\mathcal{K},f,g)= \frac{1}{2}\sum _{n=0}^N (-1)^n n \zeta'(0),
\end{eqnarray}
where $N$ is the dimension of $\mathcal{K}$.   The next proposition follows from Lemma~\ref{le-5.a}  immediately. 

\begin{proposition}\label{pr-6.f}
For any simplicial complex $\mathcal{K}$, any real-valued weights $f$ and $g$   on the set $V$,  and any non-vanishing real function $h$ on $V$,   we  have 
\begin{eqnarray*}
T(\mathcal{K},hf,hg)=T(\mathcal{K},f,g). 
\end{eqnarray*}
\end{proposition} 

\begin{proof}
By Lemma~\ref{le-5.a}  and \cite[Lemma~2.2]{dig1},  we have 
\begin{eqnarray*}
\Delta_n^{f,g}=\varphi^{-1}\circ \Delta_n^{hf,hg}\circ \varphi
\end{eqnarray*}
whiich  implies that 
all the eigenvalues (counted with multiplicities)  of   $\Delta_n^{f,g}$  for $\mathcal{K}_n^{hf,hg}$ are equal.  By the definition (\ref{eq-t-def}) of the weighted analytic torsions,  the proposition follows.  
\end{proof}

We fix $f$ and multiply $g$ by a non-zero scalar $c$.  By applying \cite[Corollary~3.8]{lin1} to the R-torsion,  we  have

\begin{proposition}\label{pr-3.v}
Let $\mathcal{K}$ be a simplicial complex of dimension $N$ with the set of   its  vertices  $V$.     Let $f$ and $g$ be two real-valued weights on $V$ such that for any $v\in V$,  $g(v)\neq 0$  implies $f(v)\neq 0$.   Let $c\neq 0$  be a real constant.  Then 
\begin{eqnarray}\label{eq-8.x}
T(\mathcal{K},f,cg)=t(\mathcal{K},c,\epsilon(g))T(\mathcal{K},f,g)
\end{eqnarray}
where  $t(\mathcal{K},c,\epsilon(g))=|c|^{s(\mathcal{K},\epsilon(g))}$  and  
\begin{eqnarray*}
s(\mathcal{K},\epsilon(g))=\sum_{n=0}^N(-1)^n \dim \partial_n C_n(\mathcal{K}^\times_{\epsilon(g)};\mathbb{R}).
\end{eqnarray*}
\end{proposition}

\begin{proof}
Since for any $v\in V$,  $g(v)\neq 0$  implies $f(v)\neq 0$,  we have 
that both $f$  and $g$  are  non-vanishing on the set of the vertices of $\mathcal{K}^\times _g$.  Hence the ratio function $g/f$ is  well-defined on $\mathcal{K}^\times_g$.  Consequently,  with the help of   Theorem~\ref{th-isom},  $H_n(\mathcal{K}^\times _g,f,g;\mathbb{R})$  is linearly isometric   to $H_n(\mathcal{K}^\times _g,1,g/f;\mathbb{R})$,  which is linearly isomorphic to $H_n(\mathcal{K}^\times_g;\mathbb{R})$. 
 For any $n\geq 0$ and any   $a,b\in C_n(\mathcal{K}^\times_g;\mathbb{R})$, we have
\begin{eqnarray*}
\langle a,b\rangle_{cg}=c^{2(n+1)}\langle  a,b\rangle_g.   
\end{eqnarray*}
 Thus  by  \cite[Corollary~3.8]{lin1},  we have (\ref{eq-8.x}) where 
\begin{eqnarray}
t(\mathcal{K},c,\epsilon(g))&=&\prod_{n= 0}^N |c|^{(-1)^{n} (n+1) \big(\dim C_n(\mathcal{K}^\times_g;\mathbb{R})-\dim H_n(\mathcal{K}^\times_g,f,g;\mathbb{R})\big)}\nonumber\\
&=&|c|^{\sum_{n=0}^N  (-1)^n (n+1) \big(\dim C_n(\mathcal{K}^\times_g;\mathbb{R})-\dim H_n(\mathcal{K}^\times_g;\mathbb{R})\big)}\nonumber\\
&=&|c|^{\sum_{n=0}^N  (-1)^n n \big(\dim C_n(\mathcal{K}^\times_g;\mathbb{R})-\dim H_n(\mathcal{K}^\times_g;\mathbb{R})\big)}. \label{eq-revise}
\end{eqnarray}  
The last equality of (\ref{eq-revise}) follows from   the formulas for  Euler characteristic number 
\begin{eqnarray*}
\chi(\mathcal{K}^\times_g)=\sum_{n=0}^N  (-1)^n   \dim C_n(\mathcal{K}^\times_g;\mathbb{R})=\sum_{n=0}^N  (-1)^n \dim H_n(\mathcal{K}^\times_g;\mathbb{R}). 
\end{eqnarray*}
Therefore, with the help of \cite[Lemma~2.1]{dig1},  the right-hand side of  the last equality in (\ref{eq-revise})  equals to $|c|^{s(\mathcal{K},\epsilon(g))}$.  
\end{proof}

We  fix $g$  and multiply $f$  by a non-zero scalar $c$.  By a straight-forward calculation of the analytic torsion,  we  have 

\begin{proposition}\label{pr-add}
Let $\mathcal{K}$ be a simplicial complex of dimension $N$ with the set of   its  vertices  $V$.     Let $f$ and $g$ be a real-valued weights on $V$ such that for any $v\in V$, $g(v)\neq 0$  implies $f(v)\neq 0$.   Let $c\neq 0$  be a real constant.  Then 
\begin{eqnarray}\label{eq-8.y}
T(\mathcal{K},cf,g)=t(\mathcal{K},c,\epsilon(g))^{-1}T(\mathcal{K},f,g). 
\end{eqnarray}
\end{proposition}

\begin{proof}
The proof is an analogue of the proof for \cite[Proposition~3.7]{dig1}.  Let $n\geq 0$. 
It follows directly that 
\begin{eqnarray*}
\partial_n^{cf}=c\partial_n^f,~~~
(\partial_{n+1}^{cf}\mid_{\mathcal{K}^\times_g})^* =c (\partial_{n+1}^f\mid_{\mathcal{K}^\times_g})^*. 
\end{eqnarray*}
Thus we have 
\begin{eqnarray*}
\Delta_n^{cf,g}=c^2\Delta_n^{f,g}.  
\end{eqnarray*}
Consequently,  with the help of \cite[(3.21)]{lin1},     
\begin{eqnarray*}
\log T(\mathcal{K},cf, g)
&=&\frac{1}{2}\sum_{n=0}^N(-1)^n n 
\big(-\sum_{\lambda_i>0}\log (\lambda_i)
\big)\\
&&-\frac{1}{2}\sum_{n=0}^N(-1)^n n (2\log |c|)\big(\dim C_n(\mathcal{K}_g^\times;\mathbb{R})-\dim H_n(\mathcal{K}_g^\times;\mathbb{R})\big)\\
&=&\frac{1}{2}\sum_{n=0}^N(-1)^n n \frac{d}{ds}\Big| _{s=0}\Big(\sum_{\lambda_i>0}\frac{1}{\lambda_i^s}\Big)
-(\log |c|)s(\mathcal{K},\epsilon(g))\\
&=& \log T(\mathcal{K},f,g)-(\log |c|)s(\mathcal{K},\epsilon(g)). 
\end{eqnarray*}
Taking the exponential map on both sides of the  equations,  we have 
(\ref{eq-8.y}).     
\end{proof}

Finally,  summarizing  Proposition~\ref{pr-3.v}  and Proposition~\ref{pr-add},   we obtain  the proof of Theorem~\ref{th-main}. 

\begin{proof}[Proof of Theorem~\ref{th-main}]
Given a weighted simplicial complex $(\mathcal{K},f,g)$, We  consider  the new weighted simplicial complex $(\mathcal{K},f',g')$  where $f'=\epsilon(g) f$  and $g'=\epsilon(f) g$.  Then we  have 
\begin{eqnarray*}
\epsilon(f')=\epsilon(g')=\epsilon (f)\epsilon(g). 
\end{eqnarray*} 
Applying Proposition~\ref{pr-6.f}  to $(\mathcal{K},f'g')$,  we obtain Theorem~\ref{th-main}~(i).  Applying Proposition~\ref{pr-3.v}  and Proposition~\ref{pr-add} to $(\mathcal{K},f',g')$,  we obtain Theorem~\ref{th-main}~(ii).  
\end{proof}

\smallskip

For any chain complex  $\{C_n\}_{n=0}^N$ where each $C_n$ has an inner product,  the analytic torsion can be obtained  by calculating  the R-torsion (cf. \cite[Section~3]{lin1}).  A complete construction of the R-torsion has been given in \cite[Section~3]{lin1} and reviewed in \cite[Section~2]{dig1} hence it will not be  restated here.   Particularly, our weighted analytic torsions for  weighted simplicial complexes   can be obtained  by  calculating the corresponding  R-torsions.  In the remaining part of this paper, 
we give some examples  of weighted analytic torsions for weighted simplicial complexes by calculating the R-torsions.  The calculations here are similar with \cite[Section~4]{dig1}.  

\begin{example}\label{ex-1}
Let $V=\{v_0,v_1,v_2,v_3\}$. Let $f$  and $g$ be two real-valued non-vanishing functions on $V$ such that (i). $f$ is non-vanishing on $V$; (ii).  $g(v_0),g(v_1),g(v_2)\neq 0$ and  $g(v_3)=0$.  In the following, we use  $v_0$ to denote a $0$-simplex $\{v_0\}$, use $v_0v_1$ to denote a $1$-simplex $\{v_0,v_1\}$, etc. 
\begin{enumerate}[(1).]
\item
 Let 
\begin{eqnarray*}
\mathcal{K}=\{v_0,v_1,v_2,v_0v_1,v_1v_2,v_0v_2,v_0v_1v_2,v_0v_3,
v_1v_3,v_2v_3\}
\end{eqnarray*}
   Then 
\begin{eqnarray*}
\mathcal{K}^\times_{\epsilon(f)\epsilon(g)}=\mathcal{K}^\times_g=\{v_0,v_1,
v_2,v_0v_1,v_1v_2,v_0v_2,v_0v_1v_2\}
\end{eqnarray*}  
is a  solid triangle. 
Similar with the calculation in \cite[Example~3.10]{lin1}  and  \cite[Example~4.2]{dig1},  
we have 
\begin{eqnarray*}
T(\mathcal{K},f,g)=T(\mathcal{K}^\times_g,f,g)=\sqrt{\sum_{i=0}^2\frac{f(v_i)^2}{g(v_i)^2}}. 
\end{eqnarray*}
\item
Let
\begin{eqnarray*}
\mathcal{L}=\{v_0,v_1,v_2,v_0v_1,v_1v_2,v_0v_2,v_0v_3,
v_1v_3,v_2v_3\}. 
\end{eqnarray*}
  Then 
\begin{eqnarray*}
\mathcal{L}^\times_{\epsilon(f)\epsilon(g)}=\mathcal{L}^\times_g= \{v_0,v_1,v_2,v_0v_1,v_1v_2,v_0v_2\}.  
\end{eqnarray*}
We choose the following $\langle~,~\rangle_g$-orthonormal bases 
\begin{eqnarray*}
\omega_0=\Big\{\frac{v_0}{g(v_0)},\frac{v_1}{g(v_1)},\frac{v_2}{g(v_2)}\Big\}
\end{eqnarray*}
in $C_0(\mathcal{L}^\times_g;\mathbb{R})$   and  
\begin{eqnarray*}
\omega_1=\Big\{\frac{v_0v_1}{g(v_0)g(v_1)},\frac{v_1v_2}{g(v_1)g(v_2)},\frac{v_0v_2}{g(v_0)g(v_2)}\Big\}
\end{eqnarray*}
in $C_1(\mathcal{L}^\times_g;\mathbb{R})$.  Choose also the basis
\begin{eqnarray*}
b_0=\Big\{\frac{f(v_0)v_1-f(v_1)v_0}{g(v_0)g(v_1)},\frac{f(v_1)v_2-f(v_2)v_1}{g(v_1)g(v_2)}\Big\}
\end{eqnarray*} 
in $\partial_1^fC_1(\mathcal{L}^\times_g;\mathbb{R})$.  Then  its lift is 
\begin{eqnarray*}
\tilde b_1 =\Big\{\frac{v_0v_1}{g(v_0)g(v_1)},\frac{v_1v_2}{g(v_1)g(v_2)}\Big\}.  
\end{eqnarray*}  
Taking the orthogonal complement of   $\partial_1^f C_1(\mathcal{L}_g^\times;\mathbb{R})$  in ${\rm Ker}\partial_0^f$  with respect to $\langle~,~\rangle_g$,  we have 
\begin{eqnarray*}
h_0=\Big\{  \frac{1}{\sqrt{\sum_{i=0}^2\frac{f(v_i)^2}{g(v_i)^2}}} \Big(\frac{f(v_0)}{g(v_0)^2}v_0 +\frac{f(v_1)}{g(v_1)^2}v_1       + \frac{f(v_2)}{g(v_2)^2}v_2 \Big) \Big\}.
\end{eqnarray*}
Moreover,  note that since  $C_2(\mathcal{L}^\times_g;\mathbb{R})=0$,   we have  $\partial_2 C_2(\mathcal{L}^\times_g;\mathbb{R})=0$,  which implies  $b_1=0$.   Hence 
\begin{eqnarray*}
h_1=\Big\{\frac{f(v_0){v_1v_2}-f(v_1)v_0v_2+ f(v_2)v_0v_1 }{g(v_0)g(v_1)g(v_2)} 
  \Big\}.  
\end{eqnarray*}
Thus 
\begin{eqnarray*}
[b_0,h_0,\tilde b_0/\omega_0]=\Big|\frac{f(v_1)}{g(v_1)}\Big|\sqrt{\sum_{i=0}^2\frac{f(v_i)^2}{g(v_i)^2}}  
\end{eqnarray*}
and 
\begin{eqnarray*}
[b_1,h_1,\tilde b_1/\omega_1]=\frac{|f(v_1)|}{|g(v_1)|}. 
\end{eqnarray*}
Therefore, 
\begin{eqnarray*}
T(\mathcal{L},f,g)=T(\mathcal{L}^\times_g,f,g)
=\sqrt{\sum_{i=0}^2\frac{f(v_i)^2}{g(v_i)^2}}.  
\end{eqnarray*}
\end{enumerate}
\end{example}

\begin{example}\label{ex-2}
Let $V=\{v_0,v_1,v_2,v_3,v_4\}$.   Suppose both $f$  and $g$   are  non-vanishing functions on $V$. 
\begin{enumerate}[(1).]
\item 
 Let 
\begin{eqnarray*}
\mathcal{K}=\{v_0,v_1,v_2,v_3,v_4,v_0v_1,v_1v_2,v_2v_3,v_3v_4\}. 
\end{eqnarray*}
 Then similar with the calculation in \cite[Example~3.9]{lin1}  and  \cite[Example~4.1]{dig1},   we  have 
\begin{eqnarray*}
T(\mathcal{K},f,g)=\frac{1}{\sqrt{\sum_{i=0}^3\frac{f(v_i)^2}{g(v_i)^2}}} \Big|  \sum_{k=0}^3   \frac{f(v_k)}{g(v_k)}\prod_{0\leq j\leq k-1}  \frac{f(v_{j+1})}{g(v_{j+1})}\prod_{k\leq j\leq 2} \frac{f(v_j)}{g(v_j)}\Big|.
\end{eqnarray*}
\item
Let 
\begin{eqnarray*}
\mathcal{L}=\{0,1,2,3,01,12,23,02,03,13\}. 
\end{eqnarray*}
We choose the following $\langle~,~\rangle_g$-orthonormal bases
\begin{eqnarray*}
\omega_0=\Big\{\frac{v_0}{g(v_0)},\frac{v_1}{g(v_1)},\frac{v_2}{g(v_2)}, \frac{v_3}{g(v_3)}\Big\}
\end{eqnarray*}
in $C_0(\mathcal{L};\mathbb{R})$   and  
\begin{eqnarray*}
\omega_1=\Big\{\frac{v_0v_1}{g(v_0)g(v_1)},\frac{v_1v_2}{g(v_1)g(v_2)},\frac{v_2v_3}{g(v_2)g(v_3)}, \frac{v_0v_2}{g(v_0)g(v_2)}, \frac{v_0v_3}{g(v_0)g(v_3)}, \frac{v_1v_3}{g(v_1)g(v_3)}\Big\}
\end{eqnarray*}
in $C_1(\mathcal{L};\mathbb{R})$. 
Choose also the basis 
\begin{eqnarray*}
b_0=\Big\{\frac{f(v_0)v_1-f(v_1)v_0}{g(v_0)g(v_1)}, \frac{f(v_1)v_2-f(v_2)v_1}{g(v_1)g(v_2)}, \frac{f(v_2)v_3-f(v_3)v_2}{g(v_2)g(v_3)}   \Big\}
\end{eqnarray*}
in $\partial^f_1 C_1(\mathcal{L};\mathbb{R})$  and set its lift 
\begin{eqnarray*}
\tilde b_1= \Big\{\frac{v_0v_{1}}{g(v_0)g(v_{1})}, \frac{v_1v_{2}}{g(v_1)g(v_{2})},  \frac{v_2v_{3}}{g(v_3)g(v_{3})}  \Big\}. 
\end{eqnarray*}
Moreover,  note that $\tilde b_0=0$.   Taking the orthogonal complement of $\partial_1^fC_1(\mathcal{L};\mathbb{R})$  in ${\rm Ker}\partial_0^f= C_0(\mathcal{L};\mathbb{R})$ with respect to $\langle~,~\rangle_g$,  we have 
\begin{eqnarray*}
h_0=\Big\{\frac{\sum_{i=0}^3f(v_i)v_i}{\sqrt{\sum_{i=0}^3 f(v_i)g(v_i)^2}} \Big\}. 
\end{eqnarray*}
Note that   $b_1=0$.  Thus 
\begin{eqnarray*}
h_1=\Big\{  \frac{v_0v_2}{g(v_0)g(v_2)}, \frac{v_0v_3}{g(v_0)g(v_3)},   \frac{v_1v_3}{g(v_1)g(v_3)} \Big\}. 
\end{eqnarray*}
Thus 
\begin{eqnarray*}
[b_0,h_0,\tilde b_0/\omega_0]=\frac{1}{\sqrt{\sum_{i=0}^3\frac{f(v_i)^2}{g(v_i)^2}}} \Big|  \sum_{k=0}^3   \frac{f(v_k)}{g(v_k)}\prod_{0\leq j\leq k-1}  \frac{f(v_{j+1})}{g(v_{j+1})}\prod_{k\leq j\leq 2} \frac{f(v_j)}{g(v_j)}\Big|
\end{eqnarray*}
and 
\begin{eqnarray*}
[b_1,h_1,\tilde b_1/\omega_1]=1. 
\end{eqnarray*}
Therefore, 
\begin{eqnarray*}
T(\mathcal{L},f,g)=\frac{1}{\sqrt{\sum_{i=0}^3\frac{f(v_i)^2}{g(v_i)^2}}} \Big|  \sum_{k=0}^3   \frac{f(v_k)}{g(v_k)}\prod_{0\leq j\leq k-1}  \frac{f(v_{j+1})}{g(v_{j+1})}\prod_{k\leq j\leq 2} \frac{f(v_j)}{g(v_j)}\Big|. 
\end{eqnarray*}

\item
Let 
\begin{eqnarray*}
\mathcal{M}=\{v_0,v_1,v_2,v_3,v_4,v_0v_1,v_0v_2,v_0v_3,v_0v_4\}. 
\end{eqnarray*}
We choose the $\langle~,~\rangle_g$-orthonormal bases 
\begin{eqnarray*}
\omega_0=\Big\{  \frac{v_i}{g(v_i)}  \mid  0\leq i\leq 4 \Big\}
\end{eqnarray*}
in $C_0(\mathcal{L};\mathbb{R})$ and 
 \begin{eqnarray*}
\omega_1=\Big\{  \frac{v_0v_i}{g(v_0)g(v_i)}  \mid  1\leq i\leq 4 \Big\}
\end{eqnarray*}
in $C_1(\mathcal{M};\mathbb{R})$.  Let 
\begin{eqnarray*}
b_0=\Big\{\frac{f(v_0)v_i-f(v_i)v_0}{g(v_0)g(v_i)}\mid 1\leq i\leq 4 \Big\}. 
\end{eqnarray*}
be a basis in $\partial_1^f(C_1(\mathcal{M};\mathbb{R}))$.  The lift of $b_0$ in $C_1(\mathcal{M};\mathbb{R})$  is 
\begin{eqnarray*}
\tilde b_1=\omega_1. 
\end{eqnarray*}
By taking the $\langle~,~\rangle_g$-orthogonal complement of $\partial_1^f(C_1(\mathcal{M};\mathbb{R}))$  in $C_0(\mathcal{L};\mathbb{R})$,  we  have 
\begin{eqnarray*}
h_0= \frac{1}{\sqrt{\sum_{i=0}^4 \frac{f(v_i)^2}{g(v_i)^2}}}\Big(\sum_{i=0}^4  \frac{f(v_i)}{g(v_i)}\frac{v_i}{g(v_i)}\Big). 
\end{eqnarray*}
Moreover,
\begin{eqnarray*}
\tilde b_0=0,~~~  b_1=0, ~~~ h_1=0. 
\end{eqnarray*}
It follows by a calculation similar with \cite[Example~3.9]{lin1}  and \cite[Example~4.1]{dig1}  that 
\begin{eqnarray*}
[b_0,h_0,\tilde b_0/\omega_0]=\frac{f(v_0)^3}{g(v_0)^3}\sqrt{\sum_{i=0}^4 \frac{f(v_i)^2}{g(v_i)^2} }
\end{eqnarray*}
and 
\begin{eqnarray*}
[b_1,h_1,\tilde b_1/\omega_1]=1.  
\end{eqnarray*}
Consequently,  
\begin{eqnarray*}
T(\mathcal{M},f,g)=\frac{f(v_0)^3}{g(v_0)^3}\sqrt{\sum_{i=0}^4 \frac{f(v_i)^2}{g(v_i)^2} }.   
\end{eqnarray*}
\end{enumerate}
\end{example}

\bigskip

{\bf Acknowledgements}. The authors would like to express their deep gratitude to Professor Yong Lin and Professor Jie Wu  for their helpful guidance and kind encouragement.

\bigskip

\bigskip

Shiquan Ren

Address: Yau Mathematical Sciences Center, Tsinghua University, Beijing 100084, China. 

E-mail:  srenmath@126.com

\bigskip

Chengyuan Wu

Address: Institute of High Performance Computing, A*STAR,  Singapore 138632, Singapore. 

E-mail: wuchengyuan@u.nus.edu
 

\begin{thebibliography}{99}


 
\bibitem{ch}
J. Cheeger, Analytic torsion and the heat equation,  \emph{Annals of Mathematics},  {\bf 109}, 259-322, 1979. 

\bibitem{1990}
Robert. J. MacG. Dawson, Homology of weighted simplicial complexes, \emph{Cahiers de Topologie
et G\'{e}om\'{e}trie Diff\'{e}rentielle Cat\'{e}goriques}, {\bf 31} no. 3,   229-243,  1990.


\bibitem{fried}
David Fried, Analytic torsion and closed geodesics on hyperbolic manifolds, \emph{Inventions Mathematicae}, {\bf 84}, 523-540, 1986. 



\bibitem{lin1}
Alexander Grigor'yan,  Yong  Lin, Shing-Tung Yau,  Torsion of digraphs and path complexes,  arXiv: 2012.07302v1,  2020. 

\bibitem{lin2}
Alexander Grigor'yan,  Yong  Lin, Yuri Muranov,  Shing-Tung Yau,  Homologies of path complexes and digraphs,  arXiv: 1207.2834,  2013. 

\bibitem{lin3}
Alexander Grigor'yan,  Yong  Lin, Yuri Muranov,  Shing-Tung Yau,  Homotopy  theory for  digraphs,  \emph{Pure and Applied Mathematics Quarterly},  {\bf 10} (4),   619-674,  2014. 

  
\bibitem{lin4}
Alexander Grigor'yan,  Yong  Lin, Yuri Muranov,  Shing-Tung Yau,  Cohomology of digraphs and (undirected) graph,  \emph{Asian Journal of  Mathematics},  {\bf 15} (5),   887-932,  2015. 

  
\bibitem{lin5}
Alexander Grigor'yan,   Yuri Muranov,  Shing-Tung Yau,  Homologies of digraphs and K\"{u}nneth formulas,  \emph{Communications in Analysis and Geometry},  {\bf 25},   969-1018,  2017. 

\bibitem{lin6}
Alexander Grigor'yan,   Yuri Muranov,  Shing-Tung Yau,  Path complexes and their homologies,  \emph{Journal of Mathematical Sciences},  {\bf 248} (5),   564-599,  2020. 


\bibitem{xia1}
 Zhenyu Meng, D Vijay Anand, Yunpeng Lu, Jie Wu, Kelin Xia, Weighted persistent homology for biomolecular data analysis,  \emph{Scientific Reports}, {\bf 10}, no. 2079, 2020.

 
\bibitem{milnor}
John Milnor,  Whitehead torsion, \emph{Bulletin of the American Mathematical Society},  {\bf 72},   358-426, 1966. 

\bibitem{muller}
Werner  M\"{u}ller,  Analytic torsion and $R$-torsion of Riemannian manifolds,  \emph{Advance in Mathematics}, {\bf 28},  233-305, 1978. 

 

\bibitem{ray}
D. B. Ray,  I.  Singer,  R-torsion and Laplacian on Riemannian manifolds, \emph{Advance in Mathematics}, {\bf 7}, 145-210, 1971. 
 
\bibitem{rei}
K. Reidemeister, Homotopieringe and Linsenr\"{a}ume, \emph{Hamburger Abhaudl}, {\bf 11}, 102-109, 1935. 

\bibitem{chengyuan1}
Shiquan  Ren,  Chengyuan Wu, Jie Wu,  Weighted persistent homology,  \emph{Rocky Mountain Journal of  Mathematics}, {\bf 48},  no. 8,  2661-2687, 2018.  

\bibitem{chengyuan11}
Shiquan   Ren,  Chengyuan Wu, Jie Wu,  Computational tools in weighted persistent homology, \emph{Chinese Annals of Mathematics Series B}, {\bf 42}, no.  2,  237-258, 2021. 

 
\bibitem{dig1}
Shiquan  Ren,  Chong Wang, Weighted analytic torsions for weighted digraphs,  arXiv:  2103.09552,  2021. 


\bibitem{chengyuan2}
Chengyuan  Wu,  Shiquan Ren,  Jie Wu,   Kelin   Xia,  Discrete Morse theory for weighted simplicial complexes,  \emph{Topology and Its  Applications}, {\bf 270}, no. 107038, 2020.  

 
\bibitem{chengyuan3}
Chengyuan  Wu, Shiquan  Ren, Jie Wu,  Kelin   Xia,  Weighted (co)homology and weighted Laplacian,  arXiv: 1804.06990. 
 

\bibitem{chengyuan5}
Chengyuan  Wu, Shiquan  Ren, Jie Wu,  Kelin  Xia,   Weighted fundamental group, \emph{Bulletin of the Malaysian Mathematical Sciences Society},  {\bf 43}, 4065-4088,  2020. 




  \end{thebibliography}
\end{document}